\documentclass[oneside,american,english,british]{amsart}
\usepackage[T1]{fontenc}
\usepackage{geometry}
\usepackage{amsthm}
\usepackage{amstext}
\usepackage{amssymb}
\usepackage{setspace}
\usepackage{units}
\usepackage{babel}
\usepackage[latin9,cp1255]{inputenc}
\PassOptionsToPackage{normalem}{ulem}
\usepackage{ulem}

\makeatletter
\numberwithin{equation}{section}
\numberwithin{figure}{section}

\theoremstyle{plain}
\newtheorem{thm}{\protect\theoremname}
  \theoremstyle{definition}
  \newtheorem{defn}[thm]{\protect\definitionname}
  \theoremstyle{plain}
  \newtheorem{prop}[thm]{\protect\propositionname}
  \theoremstyle{plain}
  \newtheorem{lem}[thm]{\protect\lemmaname}
  \theoremstyle{plain}
  \newtheorem{cor}[thm]{\protect\corollaryname}
  \theoremstyle{remark}
  \newtheorem{rem}[thm]{\protect\remarkname}
  \theoremstyle{definition}
  \newtheorem{example}[thm]{\protect\examplename}
  \theoremstyle{remark}
  \newtheorem{claim}[thm]{\protect\claimname}
  \theoremstyle{plain}
  \newtheorem{fact}[thm]{\protect\factname}
\numberwithin{thm}{section}

\makeatother

\usepackage{babel}
  \addto\captionsamerican{\renewcommand{\claimname}{Claim}}
  \addto\captionsamerican{\renewcommand{\corollaryname}{Corollary}}
  \addto\captionsamerican{\renewcommand{\definitionname}{Definition}}
  \addto\captionsamerican{\renewcommand{\examplename}{Example}}
  \addto\captionsamerican{\renewcommand{\factname}{Fact}}
  \addto\captionsamerican{\renewcommand{\lemmaname}{Lemma}}
  \addto\captionsamerican{\renewcommand{\propositionname}{Proposition}}
  \addto\captionsamerican{\renewcommand{\remarkname}{Remark}}
  \addto\captionsamerican{\renewcommand{\theoremname}{Theorem}}
  \addto\captionsbritish{\renewcommand{\claimname}{Claim}}
  \addto\captionsbritish{\renewcommand{\corollaryname}{Corollary}}
  \addto\captionsbritish{\renewcommand{\definitionname}{Definition}}
  \addto\captionsbritish{\renewcommand{\examplename}{Example}}
  \addto\captionsbritish{\renewcommand{\factname}{Fact}}
  \addto\captionsbritish{\renewcommand{\lemmaname}{Lemma}}
  \addto\captionsbritish{\renewcommand{\propositionname}{Proposition}}
  \addto\captionsbritish{\renewcommand{\remarkname}{Remark}}
  \addto\captionsbritish{\renewcommand{\theoremname}{Theorem}}
  \addto\captionsenglish{\renewcommand{\claimname}{Claim}}
  \addto\captionsenglish{\renewcommand{\corollaryname}{Corollary}}
  \addto\captionsenglish{\renewcommand{\definitionname}{Definition}}
  \addto\captionsenglish{\renewcommand{\examplename}{Example}}
  \addto\captionsenglish{\renewcommand{\factname}{Fact}}
  \addto\captionsenglish{\renewcommand{\lemmaname}{Lemma}}
  \addto\captionsenglish{\renewcommand{\propositionname}{Proposition}}
  \addto\captionsenglish{\renewcommand{\remarkname}{Remark}}
  \addto\captionsenglish{\renewcommand{\theoremname}{Theorem}}
  \providecommand{\claimname}{Claim}
  \providecommand{\corollaryname}{Corollary}
  \providecommand{\definitionname}{Definition}
  \providecommand{\examplename}{Example}
  \providecommand{\factname}{Fact}
  \providecommand{\lemmaname}{Lemma}
  \providecommand{\propositionname}{Proposition}
  \providecommand{\remarkname}{Remark}
\providecommand{\theoremname}{Theorem}

\begin{document}
\begin{onehalfspace}

\title{{\large Hjorth Analysis of General Polish Group Actions}}
\author{Ohad Drucker}
\begin{abstract}

Hjorth has introduced a Scott analysis for general Polish group actions,
and has asked whether his notion of rank satisfies a boundedness principle
similar to the one of Scott rank - namely, the orbit equivalence relation
is Borel if and only if Hjorth ranks are bounded.

We present the principles of Hjorth analysis and Hjorth rank, and answer Hjorth's question positively. From that we get a positive answer to a conjecture due to Hjorth - for every limit ordinal $\alpha$ , the set of elements whose orbit is of complexity less than $\alpha$  is a Borel set. We then show Nadel's theorem for Hjorth
rank - the rank of $x$ is no more than $\omega_{1}^{ck(x)}$.

\end{abstract}
\maketitle

\section{Introduction}

In 1965, Scott has introduced a method for completely characterizing
a countable model by formulas of $\mathcal{L}_{\omega_{1},\omega}$.
This method has quickly opened the way to a better understanding of
isomorphism of countable models. To name a few examples, it turned
out that the isomorphism relation of models of a theory $T$ is a well behaved and well understood ``limit'' of Borel equivalence relations, and that this isomorphism relation is Borel if and only if the ranks of the characterizing formulas of the $T$ - models are uniformly bounded.

The isomorphism of countable models of a theory $T$ is just another example of an
orbit equivalence relation induced by a Polish action, namely the
action of $S_{\infty}$ on the collection of countable models, $Mod_{\mathcal{L}}{(T)}.$
Hence, a natural question arises - can a similar method be developed
for the general scenario ? Since Scott analysis heavily uses the internal
structure of the points of $Mod_{\mathcal{L}}{(T)}$, it was very unclear how the
general method will look like.

A substantial progress was achieved by Hjorth in 2000, introducing
a Scott analysis for actions of $S_{\infty}$ {\cite{key-1}}. The same idea was
developed by him in \cite{key-10,key-2} to form a Scott analysis for general Polish group actions,
but the work was never published.

In what follows, we review Hjorth's work and continue it, showing
that Hjorth has indeed found a decent Scott analysis for general Polish
group actions. We prove various interesting properties of the new Hjorth analysis.

\bigskip{}

Let us begin with a general description of the main results, which will not be complete without the following definition:

\begin{defn} \label{Definition - Scott analysis}

A \textit{Scott analysis of Polish actions} is a method defining for each Polish
$G$ - space $X$ a decreasing sequence of equivalence relations $\equiv_{\alpha}$
and for each $x\in X$ a countable ordinal $\delta(x)$ such that:

\begin{enumerate}

\item $\equiv_{\alpha}$ are Borel and invariant under $G$.
\selectlanguage{english}%
\item $E_{G}^{X}=\bigcap_{\alpha<\omega_{1}}\equiv_{\alpha}.$
\selectlanguage{british}%
\item The function $\delta:X\to(\omega_{1},<)$ is Borel and invariant under
the action of $G$ .
\item There is an $\alpha<\omega_{1}$ such that for every $x,y\in X$,
$x$ and $y$ are orbit equivalent if and only if $x\equiv_{\delta(x)+\alpha}y.$

\end{enumerate}
\end{defn}

The first thing to be shown is that there is a Scott analysis of Polish actions. Sections 3 and 4 , due to Hjorth, present the outlines of his method
and show that it is indeed a Scott analysis of Polish actions. The construction relies on a relation $\leq_{\alpha}$ which is
non-symmetric and transitive. The relation is between pairs, each
pair has an object of the Polish space and an open set of $G$. When
we step up the ordinals, $(x,U)\leq_{\alpha}(y,V)$ is trying to tell
us more about how does the two actions, of $U$ on $x$ and of $V$
on $y$ , compare to each other. We then define the equivalence relation
$\equiv_{\alpha}$ . The definition is what we believe to be a somewhat
improved version of the one originally given by Hjorth. Roughly speaking, $x$ and $y$ are
$\alpha$ equivalent if they belong to the same $\mathbf\Pi^0_\alpha$ invariant sets, 
although that is precisely true only for limit $\alpha$'s. It is then shown that
$\equiv_{\alpha}$ is Borel and invariant under $G$, and that the
intersection of all $\equiv_{\alpha}$ is the orbit equivalence relation.
That leads to a different proof of the well known theorem of Sami - if all orbits are $\mathbf\Pi^0_\alpha$ then $E^X_G$ is Borel.

Next, Hjorth rank $\delta(x)$ is defined. The definition differs
only cosmetically from the one originally given by Hjorth. It is shown
that the rank is Borel definable and invariant under $G$. A back
and forth argument proves Scott's isomorphism theorem
for that scenario - if $x\equiv_{\delta(x)+\omega}y$ then $x$ and
$y$ are orbit equivalent.

The main result of section 5 is a boundedness principle for Hjorth
rank:
\begin{thm}
Let $(G,X)$ be a Polish action and
$\mathbb{B}\subset X$ an invariant Borel set. Then $E_{G}^{\mathbb{B}}$
is Borel if and only if there is an $\alpha$ such that for every
$x\in\mathbb{B}$, $\delta(x)\leq\alpha.$
\end{thm}
This is done by first showing that sets of the form $U\cdot x$,
for $U\subseteq G$ open, are Borel, and their complexity is almost
the same as the complexity of the orbit $G\cdot x.$ It is then proved
that if for $x\in X$, the complexities of the $U\cdot x$'s are bounded,
then Hjorth rank is no higher than their bound. A proof of the boundedness
principle follows easily, as well as a Borel definition of \[\{x\ :\ [x]\ is\    \mathbf{\Pi_{\alpha}^{0}\mbox{ for }\alpha<\beta\}}\] for $\beta$ limit.  That positively answers a conjecture of Hjorth (see \cite{key-10}).The Becker Kechris theorem for the logic action is then reproved using the newly developed tools.

The last section is dedicated to the proofs of Nadel's and Sacks' theorems
for Hjorth analysis:
\begin{thm}

(Nadel) For every $x\in X$, $\delta(x)\leq\omega_{1}^{ck(x)}.$

\end{thm}
\begin{thm}

(Sacks) If for every $x$, $\delta(x)\leq\alpha$ or
$\delta(x)<\omega_{1}^{ck(x)}$ , then Hjorth ranks are bounded.

\end{thm}
The exact definition of $\omega_{1}^{ck(x)}$ will be given later. Informally, this is the first ordinal not recursive in a sequence $x_G$ that contains all the information about the action of $G$ on $x$.
The same techniques are then used to establish the complexities of Hjorth rank comparisons, such as
\[\{\ (x,y)\ :\ \delta(x)\leq\delta(y)\}\] etc.
\\

The next section intends to cover most of the knowledge this paper assumes, so that the text will be as self contained as possible. Although knowledge of forcing might help, it is not obligatory in any way - the reader can safely use the standard definition of Vaught transforms, and translate the single forcing proof of the paper to Vaught transforms' terms.

We remark that Hjorth analysis can be rephrased in terms of forcing - see the statement of lemma \ref{Main lemma section 2} for more details.
\subsection{Acknowledgments}
This paper is based on the master thesis of the author \cite{key-16}, written under the supervision of Menachem Magidor. The author wishes to thank him for his patient guidance, insightful views, and outstanding intuition, all invaluable for this work.

\section{Preliminaries}

\subsection{Descriptive Set Theory of Polish Group Actions}

For the basics of descriptive set theory and much more, the reader is encouraged
to consult \cite{key-9}.

We review basic facts about the descriptive set theory of Polish group
actions. All the following can be found in \cite{key-7,key-11}.

\selectlanguage{english}%
A \textit{Polish Topology} is a separable topology induced by a complete
metric. A \textit{Polish Space} is a topological space whose topology
is polish. A subspace of a Polish space is Polish if and only it is
$G_{\delta}$. The product of a countable collection of Polish spaces
is Polish. In particular, $\omega^{\omega}$ and $2^{\omega}$ , with
the product topology of the discrete topologies, are both Polish.
Universality properties indicate a strong connection between these
two spaces and all other Polish spaces.

A \textit{Standard Borel Space} is a set $X$ equipped with a $\sigma$-
algebra $S$ such that there is a Polish topology $\tau$ on $X$
whose Borel $\sigma$ - algebra is $S$. Given a Polish space $X$,
the \textit{Effros Borel space} of $X$ , $F(X)$ , is the set of
closed sets of $X$ with the $\sigma$ - algebra generated by
\[
\{F\in F(X)\ :\ F\cap U\neq\emptyset\}
\]
for $U\subseteq X$ open. This is a standard Borel space.

A \textit{Polish Group} is a topological group
whose topology is polish.\foreignlanguage{british}{ One important
example is $S_{\infty},$ the group of permutations of natural numbers.
This group is readily contained in $\omega^{\omega}$, and one can
easily show it is a $G_{\delta}$ subset of $\omega^{\omega}$, and
hence Polish.}

\selectlanguage{british}%
Let $G$ be a Polish group, and $H\leq G$ a subgroup. Then $H$ is
Polish if and only if it is closed. In this case, the quotient ${\nicefrac{G}{H}}$,
as the set of left cosets of $H$, is a Polish space under the quotient
topology. If $H$ is normal, it is a Polish group as well.

A continuous action of a Polish group $G$ on a Polish space $X$
\foreignlanguage{english}{is called a \textit{Polish action}. We will
say that $X$ is a \textit{Polish $G$ - space}. A polish action naturally
induces an orbit equivalence relation on $X$ , which we will denote
by $E_{G}^{X}$. For $x\in X$, the stabilizer of $x$ is $G_{x}=\{g\ :\ g\cdot x=x\}$. Its left translations $g \cdot G_x$ are all of the form $G_{x,y}=\{g\ :\ g\cdot x=y\}$ for some $y\in X$.
Since $G_{x}$ is a closed subgroup, ${\nicefrac{G}{G_x}}$ is a Polish
space. Consider the canonical bijection
\[
{\nicefrac{G}{G_x}} \to G\cdot x
\]
It is clearly continuous, so that $G\cdot x$ is the continuous injective
image of a Polish space, hence Borel. However, the canonical bijection
is a homeomorphism only when $G\cdot x$ is Polish:}

\selectlanguage{english}%
\begin{thm}

\textbf{(Effros)\label{thm:(-Effros-)} ${\nicefrac{G}{G_x}} \to G\cdot x$
}is an homeomorphism iff $G\cdot x$ is $G_{\delta}$ iff $G\cdot x$
is non - meager in its topology.
\end{thm}
\selectlanguage{british}%
An immediate corollary is that whenever an orbit is non meager, it
must be $G_{\delta}$.
\\

The orbit equivalence relation $E_{G}^{X}$ is analytic, but not always
Borel. If it is Borel, there is an $\alpha<\omega_{1}$ such that
all orbits are $\mathbf{\Pi_{\alpha}^{0}}$. It turns out that the
opposite is true as well:
\begin{thm}
(Sami) Let $X$ be a Polish $G$ - space. If there is an $\alpha<\omega_{1}$
such that all orbits are $\mathbf{\Pi_{\alpha}^{0}}$ then $E_{G}^{X}$
is Borel.\end{thm}
\begin{proof}
Define for $\alpha$ countable the following equivalence relation
$E_{\alpha}:$
\[
xE_{\alpha}y\iff\forall A\ \ \mathbf{\Pi_{\alpha}^{0}}\ invariant:\ x\in A\iff y\in A
\]
There is a universal set for $\mathbf{\Pi_{\alpha}^{0}}$ invariant
sets - namely, there is $U\subseteq\omega^{\omega}\times X$ a $\mathbf{\Pi_{\alpha}^{0}}$
set such that the set of sections $\{U_{f}\ :\ f\in\omega^{\omega}\}$
is the set of $\mathbf{\Pi_{\alpha}^{0}}$ invariant sets of $X$.
Now it is easy to see that $E_{\alpha}$ is $\mathbf{\Pi_{1}^{1}},$
and under the conditions of the theorem $E_{\alpha}=E_{G}^{X}$ for
a large enough $\alpha$.

In fact, using Louveau's separation theorem of \cite{key-13}, one
can show that the $E_{\alpha}$'s are Borel, under the additional assumption that $X$ is recursively presented.
\end{proof}
We mention another characterization of Borel orbit equivalence relations, due to Becker and Kechris:
\begin{thm}
\label{Becker Kechris 7.1.2} Let $X$ be a Polish $G$ - space. The following are equivalent:
\begin{enumerate}
\item $E_{G}^{X}$ is Borel.
\item The map $x \to G_x$ from $X$ into $F(G)$ is Borel.
\item The map $(x,y) \to G_{(x,y)}$ from $X^2$ into $F(G)$ is Borel.
\end{enumerate}
\end{thm}

\selectlanguage{english}%
\subsection{The Logic Action}
One of the most important example of Polish actions is the \textit{logic action}.
Let $\mathcal{L}$ be a countable relational language, $\mathcal{L}=(R_{i})_{i\in\omega}$
when $R_{i}$ is an $n_{i}-ary$ relation. We denote by $Mod(\mathcal{L})$
the collection of countable models, and assume that all the models
have the set of natural numbers as their universe. Then every $\mathcal{M\in}Mod(\mathcal{L})$
can in fact be coded as an element of $\Pi_{i\in\omega}2^{\omega^{n_{i}}}$
(which is homeomorphic to $2^{\omega}$). In particular, $Mod(\mathcal{L})$
inherits the topology of $\Pi_{i\in\omega}2^{\omega^{n_{i}}}$. This
is exactly the topology generated by
\[
A_{\phi,\bar{a}}=\{\mathcal{M}\ :\ \mathcal{M}\models\phi(\bar{a})\}
\]
where $\phi$ is an atomic formula or a negation of one, and $\bar{a}\in\omega^{<\omega}.$

\selectlanguage{british}%

This is indeed a very natural topology for $Mod(\mathcal{L})$ , as
the following theorem demonstrates:

\begin{thm}

(Lopez - Escobar) $B\subseteq Mod(\mathcal{L})$ is Borel invariant
if and only if there is a sentence $\phi\in\mathcal{L}_{\omega_{1},\omega}$
such that $B=Mod(\phi).$

\textup{A proof can be found in} \cite{key-7}.

\end{thm}
\selectlanguage{english}%

We now define a Polish action of $S_{\infty}$ on $Mod(L)$ by:
\[
R^{gM}(a_{1},\dots,a_{n})\iff R^{M}(g^{-1}(a_{1}),\dots,g^{-1}(a_{n}))
\]
This action is called the logic action, and it is easy to verify that
the orbit equivalence relation is exactly the isomorphism of $L$
- models.

\selectlanguage{british}%
In many cases in mathematics, we will want to study the isomorphism
classes of models of a certain first order theory, for example, the
isomorphism classes of groups or rings. We will therefore want to
restrict the action of $S_{\infty}$ only to the models of this theory.
So let $T$ be a first order theory in a countable language $\mathcal{L}$.
The collection of countable models of $T$ , $Mod_{\mathcal{L}}(T)$,
is a Borel invariant subset of $Mod(\mathcal{L})$. The continuous
action of $S_{\infty}$ on the Borel invariant subset $Mod_{\mathcal{L}}(T)$
induces the isomorphism of models of $T$ as its orbit equivalence
relation. That important example is one of the reasons we will prefer
to state our thoerem for $B$ Borel invariant and not necessarily
Polish.

\subsection{Scott Analysis}

We review the basic properties of Scott analysis, as were established
in \cite{key-14}. A more detailed review can also be found in \cite{key-4,key-7}.

\selectlanguage{english}%
\begin{defn}

Let \textbf{$\mathcal{M},\mathcal{N}\in Mod(\mathcal{L})$, $\bar{a},\bar{b}\in\omega^{<\omega}$
}of the same length.

\begin{itemize}

\item $(\mathcal{M},\bar{a})\equiv_{0}(\mathcal{N},\bar{b})$ if for every
$\phi(\bar{x})$ atomic, $\mathcal{M}\models\phi(\bar{a})\iff\mathcal{N}\models\phi(\bar{b})$.
\item $(\mathcal{M},\bar{a})\equiv_{\alpha+1}(\mathcal{N},\bar{b})$ if
for every $c\in\omega$ there is $d\in\omega$ s.t. $(\mathcal{M},\bar{a}^{\frown}c)\equiv_{\alpha}(\mathcal{N},\bar{b}^{\frown}d)$
and for every $d\in\omega$ there is $c\in\omega$ s.t. $(\mathcal{N},\bar{b}^{\frown}d)\equiv_{\alpha}(\mathcal{M},\bar{a}^{\frown}c)$.
\item For $\lambda$ limit, $(\mathcal{M},\bar{a})\equiv_{\lambda}(\mathcal{N},\bar{b})$
if for every $\alpha<\lambda$, $(\mathcal{M},\bar{a})\equiv_{\alpha}(\mathcal{N},\bar{b})$.

\end{itemize}
\end{defn}

Saying that $(\mathcal{M},\bar{a})\equiv_{\alpha}(\mathcal{N},\bar{b})$
expresses a certain similarity between the tuple $\bar{a}$ in $\mathcal{M}$
and the tuple $\bar{b}$ in $\mathcal{N}$. The similarity improves
as $\alpha$ increases. We will see that if $\alpha$ is large enough,
$(\mathcal{M},\bar{a})\simeq(\mathcal{N},\bar{b})$ , which is, there
is an isomorphism between $\mathcal{M}$ and $\mathcal{N}$ which
takes $\bar{a}$ to $\bar{b}$.

\begin{defn}

We say that $\mathcal{M}\equiv_{\alpha}\mathcal{N}$ if $(\mathcal{M},\emptyset)\equiv_{\alpha}(\mathcal{N},\emptyset).$

\end{defn}

$\equiv_{\alpha}$ is a decreasing sequence of equivalence relations.
\foreignlanguage{british}{The following is easily proved by induction:}

\begin{prop}

For every $\alpha$, $\equiv_{\alpha}$ is Borel, and in fact $\mathbf{\Pi_{1+2\cdot\alpha}^{0}.}$

\end{prop}
\selectlanguage{british}%
This equivalence relation carries information about the collection
of $\mathbf{\Pi_{\alpha}^{0}}$ invariant sets to which the models
belong, as the following theorem demonstrates:
\begin{thm}

\label{property III Scott-1}Let $A\subseteq Mod(\mathcal{L})$ be
a $\mathbf{\Pi_{\alpha}^{0}}$ invariant set. If $\mathcal{M}\equiv_{\omega\cdot\alpha}\mathcal{N}$
then $\mathcal{M}\in A\iff\mathcal{N}\in A$.

\end{thm}
Toward defining Scott rank, we need to show:
\selectlanguage{english}%
\begin{prop}

Given $\mathcal{M}\in Mod(\mathcal{L})$, there is $\alpha<\omega_{1}$
such that if $(\mathcal{M},\bar{a})\equiv_{\alpha}(\mathcal{M},\bar{b})$
then $(\mathcal{M},\bar{a})\equiv_{\alpha+1}(\mathcal{M},\bar{b})$.
Moreover, for such an $\alpha$, $(\mathcal{M},\bar{a})\equiv_{\alpha}(\mathcal{M},\bar{b})$
implies that for all $\beta$, $(\mathcal{M},\bar{a})\equiv_{\beta}(\mathcal{M},\bar{b})$.
\end{prop}
\begin{proof}

Define for every $\alpha<\omega_{1}$ : $A_{\alpha}=\{(\bar{a},\bar{b})\ :\ \neg(\mathcal{M},\bar{a})\equiv_{\alpha}(\mathcal{M},\bar{b})\}.$
This is an increasing sequence of subsets of $\omega^{<\omega}$ ,
and strictly increasing till it stabilizes. Hence, it must stabilize
at a certain point.
\end{proof}
\begin{defn}

For $\mathcal{M}\in Mod(\mathcal{L})$, $\delta(\mathcal{M})$, the
\textit{Scott rank} of $\mathcal{M}$, is the least $\alpha$ such
that for all $\bar{a},\bar{b}\in\omega^{<\omega},$ $(\mathcal{M},\bar{a})\equiv_{\alpha}(\mathcal{M},\bar{b})$
implies $(\mathcal{M},\bar{a})\equiv_{\alpha+1}(\mathcal{M},\bar{b})$.

\end{defn}

If $\mathcal{M}\equiv_{\delta(\mathcal{M})+\omega}\mathcal{N}$ then
$(\mathcal{M},\bar{a})\equiv_{\delta(\mathcal{M})}(\mathcal{N},\bar{b})$
implies $(\mathcal{M},\bar{a})\equiv_{\delta(\mathcal{M})+1}(\mathcal{N},\bar{b})$
. The main theorem follows:

\begin{thm}

\textbf{(Scott Isomorphism Theorem) }Let $\mathcal{M},\mathcal{N}\in Mod(\mathcal{L})$
such that $\mathcal{M}\equiv_{\delta(\mathcal{M})+\omega}\mathcal{N}.$
Then $\mathcal{M}\simeq\mathcal{N}$.
\end{thm}
\begin{proof}

This is done by a back \& forth argument.

\end{proof}

For a first order theory $T$ in a countable language $\mathcal{L}$,
we denote by $\simeq_{T}$ the isomorphism of models of $T$. We can
now state 3 characterizations of Borel $\simeq_{T}$:

\begin{thm}

\label{Becker Kechris Logic case}(Becker - Kechris) $\simeq_{T}$
is Borel if and only if there is an $\alpha<\omega_{1}$ such that
for every $\mathcal{M}\in Mod_{\mathcal{L}}{(T)}$, $\delta(\mathcal{M})<\alpha$

\end{thm}
\begin{thm}
\label{Becker Kechris 7.1.4}(Becker - Kechris) The following are equivalent:
\begin{enumerate}
\item $\simeq_{T}$ is Borel
\item The set $\{\ (\mathcal{M},\mathcal{N},\bar{a},\bar{b})\ :\ \mathcal{M},\mathcal{N} \in Mod_{\mathcal{L}}{(T)};\ \bar{a},\bar{b}\in \omega^{<\omega}\ of\ the\ same\ length;\ (\mathcal{M},\bar{a})\simeq\mathcal{N},\bar{b})\}$ is Borel.
\item The set $\{\ (\mathcal{M},\bar{a},\bar{b})\ :\ \mathcal{M} \in Mod_{\mathcal{L}}{(T)};\ \bar{a},\bar{b}\in \omega^{<\omega}\ of\ the\ same\ length;\ (\mathcal{M},\bar{a})\simeq\mathcal{M},\bar{b})\}$ is Borel.
\end{enumerate}
\end{thm}
\selectlanguage{british}%

A careful analysis of the lightface complexity of $\equiv_{\alpha}$
has opened way to the following result, due to Nadel \cite{key-8}:

\begin{thm}

\label{Nadel Logic Case}For every $\mathcal{M}$, $\delta(\mathcal{M})\leq\omega_{1}^{ck(\mathcal{M})}.$

\end{thm}

We recall that for $x\in2^{\omega}$, $\omega_{1}^{ck(x)}$ is the
first ordinal not computable from $x$. Any countable model $\mathcal{M}$
can be identified with an element of $2^{\omega}$, as explained above.

Sacks has shown that if for all $\mathcal{M},$ $\delta(\mathcal{M})<\omega_{1}^{ck(\mathcal{M})},$
the isomorphism relation is Borel:

\begin{thm}

\label{Sacks logic case}(Sacks) $\simeq$ is Borel if and only
if there is $\alpha<\omega_{1}$ such that for all $\mathcal{M}$,
$\delta(\mathcal{M})\leq\alpha$ or $\delta(\mathcal{M})<\omega_{1}^{ck(\mathcal{M})}$.

\textup{Proofs of all of the above theorems can be found in \cite{key-7}.}

\end{thm}

\subsection{Vaught Transforms and Forcing}

Let $\langle\mathbb{P},\leq\rangle$ be a partial order. If $p\leq q$
we say that $p$ extends $q$, and if $p,q$ have a common extension
they are \textit{compatible}. A set $D\subseteq\mathbb{P}$ is \textit{dense}
if every $p\in\mathbb{P}$ has an extension in $D$. For $\mathbb{V}$
a model of $ZFC$, we say that $G\subseteq\mathbb{P}$ is a \textit{generic
filter over $\mathbb{V}$} if: $p\in G$ and $q\geq p$ implies $q\in G,$ every pair $p,q\in G$ has a common extension in $G$, and for every
$D\subseteq\mathbb{P}$ dense, $G\cap D\neq\emptyset.$ The generic
extension of $\mathbb{V}$ by $G$ is the minimal model of $ZFC$
containing both $\mathbb{V}$ and $G$ , and is denoted by $\mathbb{V}[G]$.
For names $\tau_{1},\dots,\tau_{n}$ , a formula $\phi$ and $p\in\mathbb{P}$,
$p\Vdash\phi(\tau_{1},\dots,\tau_{n})$ if for every generic filter
$G$ such that $p\in G$, $\mathbb{V}[G]\models\phi(\tau_{1}[G],\dots,\tau_{n}[G])$.

Given a Polish space $X$, let $\mathbb{P}_{I}$ be the partial order
of non meager Borel subsets of $X$ ordered by inclusion (in fact,
$\mathbb{P}_{I}$ is a partial order of codes of non meager Borel
sets). For $G$ a generic filter in $\mathbb{P}_{I}$, there is a unique $x\in X$
in $\mathbb{V}[G]$ such that
\[
G=\{B\subseteq X\ :\ B\in\mathbb{P}_{I};\ x\in B\}
\]

so that in fact $\mathbb{V}[G]=\mathbb{V}[x]$. Let $x^{*}$ be a
canonical name for this generic element. Then for $C$ a non meager
Borel set:
\[
C\Vdash x^{*}\in\check{B}\iff C-B\ is\ meager.
\]

In fact, $\mathbb{P}_{I}$ is equivalent to Cohen forcing over $X$,
which is, forcing with the nonempty open subsets of $X$, since it
is densely embedded in the separative quotient of $\mathbb{P}_{I}.$
\begin{defn}
(Vaught Transforms) Let $X$ be a Polish $G$ - space, $A\subseteq X$
and $U\subseteq G$ open.
\[
A^{*U}=\{x\ :\ \{g\in U\ :\ g\cdot x\in A\}\ is\ comeager\ in\ U\}
\]
\[
A^{\triangle U}=\{x\ :\ \{g\in U\ :\ g\cdot x\in A\}\ is\ non\ meager\ in\ U\}
\]

We write $A^{*}$ and $A^{\triangle}$ for $A^{*G}$ and $A^{\triangle G}$,
respectively.

Using the above, it is easy to show that $x\in A^{*U}\iff U\Vdash_{\mathbb{P_{I}}}g^{*}\cdot\check{x}\in\check{A}$,
while $\mathbb{P}_{I}$ are the non meager Borel subsets of $G$,
and $g^{*}$ is the name of the generic element.

We summarize a few important properties of Vaught transforms:
\begin{lem}
Let $X$ be a Polish $G$ - space. $A,A_{n}\subseteq X$, $U\subseteq G$ open, $U_n$ a basis for the topology of $G$.
\end{lem}
\begin{enumerate}
\item $A^{\triangle}$ and $A^{*}$ are invariant, and $A$ is invariant
iff $A=A^{\triangle}$ iff $A=A^{*}.$
\item $A^{\triangle U}=X-(X-A)^{*U}.$
\item If $A=\bigcup A_{n}$ then $A^{\triangle U}=\bigcup A_{n}^{\triangle U}.$
If $A=\bigcap A_{n}$ then $A^{*U}=\bigcap A_{n}^{*U}.$
\item If $A$ is $\mathbf{\Pi_{\alpha}^{0}}$ then $A^{*U}$ is $\mathbf{\Pi_{\alpha}^{0}}$.
If $A$ is $\mathbf{\Sigma_{\alpha}^{0}}$ then $A^{\triangle U}$
is $\mathbf{\Sigma_{\alpha}^{0}}.$
\item $A^{*U}=\bigcap\{A^{\triangle U_{n}}\ :\ U_{n}\subseteq U\}.$
\end{enumerate}
\end{defn}
\subsection{Better Topologies}

Refinement of Polish topologies is a very common tool in descriptive
set theory. Given a Polish space $(X,\tau)$ and a sequence $B_{n}$
of Borel sets, there is a Polish topology refining $\tau$ such that
for all $n$, $B_{n}$ is clopen. When a Polish $G$ - space $X$
is given, refining the topology while maintaining the continuity of
the action is a harder problem. Results of Becker, Kechris and Hjorth
have led to the following:
\selectlanguage{english}%
\begin{thm}

\label{thm:(-Hjorth-) better topology}Let $(X,\tau)$ be a Polish
$G$ - space, and $U\subseteq G$ a countable collection of open sets
of $G$. $\mathcal{A}$ a countable collection of $\mathbf{\Sigma_{\alpha}^{0}(X,\tau)}.$
There is a Polish $\tau\subseteq\sigma\subseteq\mathbf{\Sigma_{\alpha}^{0}(X,\tau)}$
s.t. $(X,\sigma)$ is a Polish $G$ - action and $\mathcal{A}^{\triangle U}\subseteq\sigma$.

\end{thm}
\selectlanguage{british}%
In particular, if $X$ is a Polish $G$ - space, and $B \subseteq X$ is Borel invariant, then there is a Polish topology on $B$ such that $B$ is a Polish $G$ - space.
Proofs of the above can be found in \cite{key-7,key-11}.

\selectlanguage{american}%

\section{{\large Hjorth Analysis}}

Let $(G,X)$ be a general Polish action, and $\mathfrak{B}_{0}$ a
countable basis for the topology of $G$. Our main task in the following chapter is defining the equivalence relations $\equiv_\alpha$ which will approximate $E^{X}_{G}$, as explained in the introduction. We first define a reflexive, transitive
and non-symmetric relation between pairs of an element of $x$ and
an open subset of $G$:

\begin{defn}

For $V_{0},V_{1}\subseteq G$ open and non-empty, and $x_{0},x_{1}\in X$
:

\[
(x_{0},V_{0})\leq_{1}(x_{1},V_{1})
\]
 if
\[
\overline{V_{0}\cdot x_{0}}\subseteq\overline{V_{1}\cdot x_{1}}.
\]
At successor stages :
\[
(x_{0},V_{0})\leq_{\alpha+1}(x_{1},V_{1})
\]
 if for all $W_{0}\subseteq V_{0}$ open and non-empty, there is $W_{1}\subseteq V_{1}$
open and non-empty such that:
\[
(x_{1},W_{1})\leq_{\alpha}(x_{0},W_{0}).
\]
For $\lambda$ a limit:
\[
(x_{0},V_{0})\leq_{\lambda}(x_{1},V_{1})
\]
 if for every $\alpha<\lambda$
\[
(x_{0},V_{0})\leq_{\alpha}(x_{1},V_{1}).
\]

\end{defn}
\begin{lem}

\label{basic lemma}For $W_{0}\subseteq V_{0},$ $W_{1}\supseteq V_{1},$
all open and non-empty,

\[
(x_{0},V_{0})\leq_{\alpha}(x_{1},V_{1})
\]
implies
\[
(x_{0},W_{0})\leq_{\alpha}(x_{1},W_{1}).
\]

\end{lem}
\begin{proof}
Trivial.\end{proof}
The next lemma will be needed when proving the Borel definability of $\leq_\alpha$:
\begin{lem}

\label{countable definition <=00003D}$(x_{0},V_{0})\leq_{\alpha+1}(x_{1},V_{1})$
iff for all $W_{0}\subseteq V_{0}$ in $\mathfrak{B}_{0}$ there is
$W_{1}\subseteq V_{1}$ in $\mathfrak{B}_{0}$ with $(x_{1},W_{1})\leq_{\alpha}(x_{0},W_{0}).$
\end{lem}
\begin{proof}

Follows from lemma \ref{basic lemma}.
\end{proof}
\begin{lem}
\label{lem:(a) (b) (c)}(a) Each $\leq_{\alpha}$ is transitive.

(b) For $\alpha<\beta$ , if $(x_{0},W_{0})\leq_{\beta}(x_{1,}W_{1})$
then $(x_{0},W_{0})\leq_{\alpha}(x_{1},W_{1}).$
\begin{proof}

\textbf{(a)} Immediate.

\textbf{(b)} The cases $\beta=1$ and $\beta$ limit are obvious.
We divide the case $\beta=\gamma+1$ into 3 subcases: $\gamma=1$,
$\gamma$ is a successor and $\gamma$ is a limit.

$\gamma=1:$ Assume $(x_{0},W_{0})\leq_{2}(x_{1},W_{1})$ and $\overline{W_{0}\cdot x_{0}}\nsubseteq\overline{W_{1}\cdot x_{1}}.$
So there is an open set $O$ in $X$ that intersects $W_{0}\cdot x_{0}$
but doesn't intersect $W_{1}\cdot x_{1}$. By the continuity of the
action, we can find $U_{0}\subseteq W_{0}$ such that $\overline{U_{0}\cdot x_{0}}\subseteq O$.
Which leads to a contradiction.

The other 2 subcases involve only standard induction arguments.
\end{proof}
\end{lem}

We are now ready to define the equivalence relation:

\selectlanguage{british}%
\begin{defn}

Let $x_{0},x_{1}$ in $X$. We will say that $x_{0}\equiv_{\alpha}x_{1}$
iff for $i\in\{0,1\}$ and for $V_{i}\subseteq G$ open and nonempty,
there is $V_{1-i}\subseteq G$ open and nonempty such that $(x_{1-\text{i}},V_{1-i})\leq_{\alpha}(x_{i},V_{i}).$
\end{defn}
\selectlanguage{american}%
\begin{lem}

$\equiv_{\alpha}$ is an equivalence relation.
\end{lem}

\begin{lem}

\label{countable definition =00003D=00003D=00003D}$x_{0}\equiv_{\alpha}x_{1}$
iff for $i\in\{0,1\}$ and for $V_{i}\in\mathcal{B}_{0}$, there is
$V_{1-i}\in\mathcal{B}_{0}$ such that $(x_{1-i},V_{1-i})\leq_{\alpha}(x_{i},V_{i})$.

\begin{lem}

\label{lemma invariance under g}For every $\alpha$ and every $g\in G$
:\textbf{ }$(x_{0},V_{0})\leq_{\alpha}(g\cdot x_{0},V_{0}\cdot g^{-1})$.

\end{lem}
\end{lem}
\begin{proof}

By transfinite induction on $\alpha$. The cases $\alpha=1$ and $\alpha$
limit are trivial. For a successor $\alpha$, given $W_{0}\subseteq V_{0}$,
by the induction hypothesis we have $(g\cdot x_{0},W_{0}\cdot g^{-1})\leq_{\alpha}(x_{0},W_{0}),$
so $W_{0}\cdot g^{-1}$ is the open set we're looking for.

\end{proof}

\begin{cor}

\label{=00003D=00003D=00003D invariant under G} $\equiv_{\alpha}$
is invariant under the action of $g$ .
\end{cor}
\begin{proof}

We show that $x_{0}\equiv_{\alpha}g\cdot x_{0}$ , using Lemma \ref{lemma invariance under g}.

\end{proof}

The next task is showing that $\equiv_{\alpha}$ is Borel:

\begin{prop}

\label{complexity of <=00003D}Let $V_{n}$ be an enumeration of $\mathfrak{B}_{0}$.
Then for all $\alpha<\omega_{1}$

\end{prop}

\[
\mathfrak{R}_{\alpha}=\{(x_{0},x_{1},n,m)\ :\ (x_{0},V_{m})\leq_{\alpha}(x_{1},V_{n})\}
\]
is a $\mathbf{\Pi_{\alpha+k(\alpha)}^{0}}$ set for some $k(\alpha)\in\omega$
(and $\mathbf{\Pi_{\alpha}^{0}}$ for $\alpha$ limit).

\begin{proof}

By induction on $\alpha.$ For the case $\alpha=1$, notice that for
fixed $n,m$, the set of $x_{0},x_{1}$ such that $(x_{0},x_{1},n,m)\in\mathfrak{R}_{1}$
is $G_{\delta}$. For $\alpha$ successor, use lemma \ref{countable definition <=00003D}
.
\end{proof}
\begin{thm}

\label{Complexity of =00003D=00003D=00003D}The equivalence relation
$x_{0}\ \equiv_{\alpha}\ x_{1}$ is $\mathbf{\Pi_{\alpha+m(\alpha)}^{0}}$
(while $m(\alpha)=2$ for $\alpha$ limit) .
\end{thm}
\begin{proof}

Follows from lemma \ref{countable definition =00003D=00003D=00003D}
and the previous proposition.

\end{proof}

Hjorth has shown in \cite{key-2} that $\equiv_{\alpha}$ is potentially
$\mathbf{\Pi_{\alpha+1}^{0}}$.
\\
\\
The relations $\leq_\alpha$ can be defined in terms of Vaught transforms or forcing:
\begin{prop}
\label{Main lemma section 2}$(x,U)\leq_{\alpha}(y,W)$ if and only if for every $A$ a $\mathbf{\Pi_{\alpha}^{0}}$
set, if $y \in A^{*W}$ then $x \in A^{*U}$.\\
Restating in terms of forcing: $(x,U)\leq_{\alpha}(y,W)$ if and only if for every $A$ a $\mathbf{\Pi_{\alpha}^{0}}$ set, if $W\Vdash g^{*}y\in A$ then $U\Vdash g^{*}x\in A$ ( $g^*$ is the canonical name of the generic element ).\end{prop}
\begin{proof}
By induction over $\alpha$.

$\alpha=1:$ $(\Rightarrow)$ Assume $U\cdot x\subseteq\overline{W\cdot y}$.
Let $A$ be closed, and $W\Vdash g^{*}y\in A$, which is:
\[
W\Vdash g^{*}\in\{g\ :\ g\cdot y\in A\}
\]
This can happen only if $W-\{g\ :\ g\cdot y\in A\}$ is meager, or
equivalently in this case, empty. We can now deduce that $W\cdot y\subseteq A$
so $U\cdot x\subseteq\overline{W\cdot y}\subseteq A$. Repeating the
same argument we get $U\Vdash g^{*}x\in A$.

$(\Leftarrow)$ Let $A=\overline{W\cdot y}.$ Then $W\Vdash g^{*}y\in A$
, and hence $U\Vdash g^{*}x\in A.$ This can happen only if $U\subseteq\{g\ :\ g\cdot x\in A\}$,
so $U\cdot x\subseteq A$, as wanted.

Assume for $\beta<\alpha$, we prove for $\alpha$: $(\Rightarrow)$
Assume $(x,U)\leq_{\alpha}(y,W)$. Let $A=\bigcap_{n\in\omega}B_{n}$
for $B_{n}$ $\mathbf{\Sigma_{\beta_{n}}^{0}}$ sets , $\beta_{n}<\alpha$.
Assume $W\Vdash g^{*}y\in A$, and assume by way of contradiction
that $U\nVdash g^{*}x\in A.$ Thus, there is $U'\subseteq U$ and
$n\in\omega$ such that $U'\Vdash g^{*}x\notin B_{n}$. There is $W'\subseteq W$
such that $(y,W')\leq_{\beta_{n}}(x,U')$, and so, using the induction
hypothesis, $W'\Vdash g^{*}y\notin B_{n}.$ Contradiction.

$(\Leftarrow)$ Assume $(x,U)\nleq_{\alpha+1}(y,W)$ ( the limit case
is trivial ). There is $U'\subseteq U$ in $\mathcal{B}_{0}$ such
that for every $W'\subseteq W$ in $\mathcal{B}_{0}$ , $(y,W')\nleq_{\alpha}(x,U')$.
Using the induction hypthesis, there is $B_{W'}$ a $\mathbf{\Pi_{\alpha}^{0}}$
set such that $U'\Vdash g^{*}x\in B_{W'}$ but $W'\nVdash g^{*}y\in B_{W'}$.
We then find $W''\subseteq W'$ such that $W''\Vdash g^{*}y\notin B_{W'}$.
We denote by $A$ the set
\[
\bigcup_{W'\subseteq W;\ W'\in\mathcal{B}_{0}}(X-B_{W'})
\]
which is $\mathbf{\Pi_{\alpha+1}^{0}}.$ The above means that $W\Vdash g^{*}y\in A$.
However, $U'\Vdash g^{*}x\notin A$, so $U\nVdash g^{*}x\in A$. As
wanted.\end{proof}
\selectlanguage{british}%
\begin{rem}
Using this equivalent definition, $(x,B)\leq_{\alpha}(y,C)$ has a
meaning for any $B,C$ Borel and non - meager. At the same time, such an expansion of the definition is redundant. \end{rem}

\begin{thm}

\label{Property III}If $x \equiv_\alpha y$ then $x$ and $y$ belong to the same $\mathbf{\Pi_{\alpha}^{0}}$ invariant
sets.
\end{thm}
\begin{proof}

Assume $x\in A$ for $A$ a $\mathbf{\Pi_{\alpha}^{0}}$ invariant
set. As $A$ is invariant, $x\in A^{*G}.$ Since $x\equiv_{\alpha}y$,
there is a non empty and open $W$ such that $(y,W)\leq_{\alpha}(x,G).$
The previous lemma then implies $y\in A^{*W}.$ In particular, there
is a $g$ such that $g\cdot y\in A.$ By the invariance of $A$, $y$
must be in $A$.
\end{proof}

That leads us to the more elegant definition of $\equiv_\alpha$ for limit $\alpha$'s:
\begin{cor} For limit $\alpha$, $x\equiv_\alpha y$ if and only if $x$ and $y$ belong to the same invariant sets.
\begin{proof}
The previous theorem and theorem \ref{Complexity of =00003D=00003D=00003D}.
\end{proof}
\end{cor}

\begin{thm}
(Sami's Theorem) Let $B\subseteq X$ be an invariant Borel set,
and assume that for every $x\in B$, $[x]_{G}$ is $\mathbf{\Pi_{\alpha}^{0}}$.
Then $E_{G}^{B}$ is Borel, and in fact it is $\mathbf{\Pi_{\alpha+m(\alpha)}^{0}} \cap (B\times B)$ for some $m(\alpha) \in \omega$.
\end{thm}
\selectlanguage{british}%
\begin{proof}

The above theorem shows that in this case, $\equiv_{\alpha}$ coincides
with the orbit equivalence relation on $B$.

\end{proof}
\selectlanguage{american}%

The same argument proves the following:

\begin{cor}

Let $B\subseteq X$ be an invariant Borel set, and assume the orbit
equivalence relation $E_{G}^{B}$ is Borel. Then there is an $\alpha<\omega_{1}$
such that $E_{G}^{B}=\equiv_{\alpha}\cap\left(B\times B\right).$

\end{cor}

\section{{\large Hjorth Rank}}

We define now the Hjorth rank of an element of $X$. A careful definition
is required, since we want it to be both Borel definable and invariant
under $G$. In Scott analysis, the rank of $x$ was the first place
in which we can step up ``for free''. Here, stepping up will have
an infinitesimal price of shrinking and expanding the open sets involved.
Also, stepping up will only be allowed when the open sets are in $\mathcal{B}_{0}$.

\begin{defn}

For $x\in X$, let $\delta(x)$ be the least $\alpha$ such that for
every $V_{0},V_{1},W_{0},W_{1}$ in $\mathfrak{B}_{0}$, if
\[
\overline{W_{0}}\subseteq V_{0}\ \&\ \overline{V_{1}}\subseteq W_{1}\ \&\ (x,V_{0})\leq_{\alpha}(x,V_{1})
\]
then $(x,W_{0})\leq_{\alpha+1}(x,W_{1})$.

We claim that such a $\delta(x)$ exists, or equivalently, that there
is an ordinal with the above properties. In fact, given an $x\in X$
, $U_{1},U_{2}$ in $\mathfrak{B}_{0}$, either there is $\beta_{u_{1},u_{2}}<\omega_{1}$
such that $(x,U_{1})\leq_{\beta_{u_{1},u_{2}}}(x,U_{2})$ and $(x,U_{1})\nleq_{\beta_{u_{1},u_{2}}+1}(x,U_{2})$
, or $(x,U_{1})\leq_{\alpha}(x,U_{2})$ for every $\alpha<\omega_{1}$.
Now take $\gamma(x)=\left(sup_{U_{1},U_{2}\in\mathfrak{B}_{0}}\beta_{u_{1},u_{2}}\right)+1$.
$\gamma(x)$ has the above property.

\end{defn}

We'll call $\delta(x)$ the \textit{Hjorth rank} of $x$.

\begin{lem}

We may assume that the basis $\mathfrak{B}_{0}$ has the following
property:
\end{lem}
\begin{itemize}

\item For every $\overline{U}\subseteq W$ in $\mathfrak{B}_{0}$ and every
$g\in G$, there are $V_{1},V_{2}$ in $\mathfrak{B}_{0}$ such that
$U\cdot g\subseteq V_{1}\subseteq V_{2}\subseteq W\cdot g$ and $\overline{V_{1}}\subseteq V_{2}$.
\end{itemize}
\selectlanguage{british}%
\begin{proof}

Fix a right invariant metric $d$ (not necessarily complete), and
let $\mathfrak{B}_{0}$ be the set of balls of rational radii around
a dense set.
\end{proof}
\selectlanguage{american}%
\begin{lem}

$xE_{G}^{X}y$ implies $\delta(x)=\delta(y).$
\end{lem}
\begin{proof}

Let $\alpha$ be the rank of $x$. Assume $(g\cdot x,V)\leq_{\alpha}(g\cdot x,U),$
$\overline{O}\subseteq V$ and $\overline{U}\subseteq W$. Equivalently,
$(x,V\cdot g)\leq_{\alpha}(x,U\cdot g)$, $\overline{O\cdot g}\subseteq V\cdot g$
and $\overline{U\cdot g}\subseteq W\cdot g$. Using the previous Lemma,
there are $V_{1},V_{2},O_{1},O_{2}$ in $\mathfrak{B}_{0}$ such that
\[
U\cdot g\subseteq V_{1}\subseteq V_{2}\subseteq W\cdot g
\]
\[
\overline{V_{1}}\subseteq V_{2}
\]
\[
O\cdot g\subseteq O_{1}\subseteq O_{2}\subseteq V\cdot g
\]
\[
\overline{O_{1}}\subseteq O_{2}.
\]

Trivially, $(x,V\cdot g)\leq_{\alpha}(x,U\cdot g)$ implies $(x,O_{2})\leq_{\alpha}(x,V_{1})$.
Since $\alpha$ is the rank of $x$, we can step up to $(x,O_{1})\leq_{\alpha+1}(x,V_{2})$.
This last one trivially implies $(x,O\cdot g)\leq_{\alpha+1}(x,W\cdot g)$,
or equivalently, $(g\cdot x,O)\leq_{\alpha+1}(g\cdot x,W).$

\end{proof}
We have only allowed stepping up when the open sets are in $\mathcal{B}_{0}$.
However, this restriction does not bother us too much:
\begin{lem}

Suppose $\delta(x)=\delta$, $V_{0},V_{1}$ open and nonempty, and
$(x,V_{0})\leq_{\delta+1}(x,V_{1}).$ Then $(x,V_{0})\leq_{\delta+2}(x,V_{1})$,
and in fact, for every $\alpha,$ $(x,V_{0})\leq_{\alpha}(x,V_{1}).$
\end{lem}
\begin{proof}

Let $W_{0}\subseteq V_{0}$ in $\mathfrak{B}_{0}.$ We choose $\overline{W_{0}'}\subseteq W_{0}$
in $\mathfrak{B}_{0}.$ There is $W_{1}'\subseteq V_{1}$ in $\mathfrak{B}_{0}$
such that $(x,W_{1}')\leq_{\delta}(x,W_{0}').$ We choose $\overline{W_{1}}\subseteq W_{1}'$
in $\mathfrak{B}_{0}$ and step up to
\[
(x,W_{1})\leq_{\delta+1}(x,W_{0}).
\]

The furthermore part is proved by induction on $\alpha\geq\delta+1$.\end{proof}
\begin{rem}
The definition of $\delta(x)$ depends on the choice of the basis $\mathcal{B}_0$. However, the previous lemma shows that for a given $x \in X$, a different choice of basis changes the rank by at most $1$.
\end{rem}
\begin{lem}

\label{<=00003D goes on for another step}If $x_{0}\equiv_{\delta+1}x_{1}$
for $\delta\geq\delta(x_{1})$ then

\[
(x_{0},V_{0})\leq_{\delta+1}(x_{1},V_{1})
\]

implies $(x_{0},V_{0})\leq_{\delta+2}(x_{1},V_{1})$.

\end{lem}
\begin{proof}

Let $V_{0}'\subseteq V_{0}.$ Let $V_{1}^{*}\subseteq V_{1}$ be such
that $(x_{1},V_{1}^{*})\leq_{\delta+1}(x_{0},V_{0}')$. By transitivity
of $\leq$, $(x_{1},V_{1}^{*})\leq_{\delta+1}(x_{1},V_{1})$, hence
$(x_{1},V_{1}^{*})\leq_{\delta+2}(x_{1},V_{1})$. We can then find
$V_{1}'\subseteq V_{1}$ such that $(x_{1},V_{1}')\leq_{\delta+1}(x_{1},V_{1}^{*})\leq_{\delta+1}(x_{0},V_{0}').$
\end{proof}
\begin{lem}

Suppose that $\delta\geq\delta(x_{0}),\delta(x_{1})$ and that $x_{0}\equiv_{\delta+1}x_{1}$.
Then for $V_{0},V_{1}$ open nonempty,

\end{lem}

\[
(x_{0},V_{0})\leq_{\delta+1}(x_{1},V_{1})
\]

implies that for every $\alpha<\omega_{1}$
\[
(x_{0},V_{0})\leq_{\alpha}(x_{1},V_{1}).
\]

\begin{proof}

By induction on $\alpha\geq\delta+1$. Notice that we use the fact
that $\delta$ is higher than both $\delta(x_{0})$ and $\delta(x_{1})$.

\end{proof}

\subsection*{\textmd{We are now ready to prove the main result of this section:}}

\begin{prop}

\label{main proposition Scott}If $\delta(x_{0}),\delta(x_{1})\leq\delta$
and $x_{0}\equiv_{\delta+1}x_{1}$ , then $x_{0}$ and $x_{1}$ are
orbit equivalent.
\end{prop}
\begin{proof}

Let $W_{0}$ be open and nonempty. There is $V_{0}$ such that $(x_{0},V_{0})\leq_{\delta+1}(x_{1},W_{0})$.
Using Lemma \ref{<=00003D goes on for another step}, we will be able
to imitate the proof of Scott's isomorphism theorem.

We choose $V_{1}\subseteq V_{0}$ such that $\overline{V_{1}}\subseteq V_{0}$
and $diam(V_{1})<1$. Since $(x_{0},V_{0})\leq_{\delta+1}(x_{1},W_{0})$,
we may find $W_{1}'\subseteq W_{0}$ such that
\[
(x_{1},W_{1}')\leq_{\delta+1}(x_{0},V_{1})
\]
 ( here we use Lemma \ref{<=00003D goes on for another step} ). Then
choose $W_{1}\subseteq W_{1}'$ such that $\overline{W_{1}}\subseteq W_{1}'$
and $diam(W_{1})<1$. There is $V_{2}'\subseteq V_{1}$ such that
\[
(x_{0},V_{2}')\leq_{\delta+1}(x_{1},W_{1})
\]
where again we have used Lemma \ref{<=00003D goes on for another step}.
We continue in the same way: given
\[
(x_{0},V_{n+1}')\leq_{\delta+1}(x_{1},W_{n})
\]
choose $V_{n+1}\subseteq V_{n+1}'$ such that $\overline{V_{n+1}}\subseteq V_{n+1}'$
and $diam(V_{n+1})<\frac{1}{n+1}$ and we get $W_{n+1}'\subseteq W_{n}$
such that
\[
(x_{1},W_{n+1}')\leq_{\delta+1}(x_{0},V_{n+1}).
\]

We then find $W_{n+1}\subseteq W_{n+1}'$ with $\overline{W_{n+1}}\subseteq W_{n+1}'$
and $diam(W_{n+1})<\frac{1}{n+1}$ , and get $V_{n+2}'\subseteq V_{n+1}$
with
\[
(x_{0},V_{n+2}')\leq_{\delta+1}(x_{1},W_{n+1}).
\]

At the end of the above process we will have:

(I)

\[
\dots\subseteq V_{3}\subseteq V_{3}'\subseteq V_{2}\subseteq V_{2}'\subseteq V_{1}\subseteq V_{0}
\]

\[
\dots\subseteq W_{2}\subseteq W_{2}'\subseteq W_{1}\subseteq W_{1}'\subseteq W_{0}
\]
 such that for all $n$
\[
diam(V_{n})<\frac{1}{n};\ \ \overline{V_{n+1}}\subseteq V_{n}
\]
\[
diam(W_{n})<\frac{1}{n};\ \ \overline{W_{n+1}}\subseteq W_{n}
\]

(II) For all $n\geq1:$
\[
\overline{V_{n+1}\cdot x_{0}}\subseteq\overline{V_{n+1}'\cdot x_{0}}\subseteq\overline{W_{n}\cdot x_{1}}
\]
\[
\overline{W_{n}\cdot x_{1}}\subseteq\overline{W_{n}'\cdot x_{1}}\subseteq\overline{V_{n}\cdot x_{0}}
\]

from which we deduce
\[
\dots\subseteq\overline{V_{4}\cdot x_{0}}\subseteq\overline{W_{3}\cdot x_{1}}\subseteq\overline{V_{3}\cdot x_{0}}\subseteq\overline{W_{2}\cdot x_{1}}\subseteq\overline{V_{2}\cdot x_{0}}\subseteq\overline{W_{1}\cdot x_{1}}
\]

By (I) we can define $g$ as the only object of $\bigcap V_{n}$ and
$h$ - the only object of $\bigcap W_{n}$. Given an $\epsilon>0$
, from continuity, there is an $n\in\mathbb{N}$ such that
\[
\overline{W_{n}\cdot x_{1}}\subseteq B_{\epsilon}(h\cdot x_{1}).
\]

From (II)
\[
\overline{V_{n+1}\cdot x_{0}}\subseteq\overline{W_{n}\cdot x_{1}}
\]
 and in particular:
\[
g\cdot x_{0}\in B_{\epsilon}(h\cdot x_{1}).
\]

This is true for every $\epsilon>0$ , which is why $g\cdot x_{0}=h\cdot x_{1}.$
\end{proof}
\begin{cor}

If $\delta(x_{0}),\delta(x_{1})\leq\delta$ , $x_{0}\equiv_{\delta+1}x_{1}$
and $V_{0},W_{0}$ satisfy

\[
(x_{0},V_{0})\leq_{\delta+1}(x_{1},W_{0}).
\]

Then for every $V'\subseteq V_{0}$ open nonempty, there are $g\in V'$
and $h\in W_{0}$ such that $g\cdot x_{0}=h\cdot x_{1}.$
\end{cor}
\selectlanguage{british}%
\begin{proof}

This is in fact what we have proved now.
\end{proof}
\selectlanguage{american}%
\begin{lem}

For every $\alpha$, the set $\{x\ :\ \delta(x)\leq\alpha\}$ is $\mathbf{\Pi_{\alpha+m(\alpha)}^{0}}$,
for $m(\alpha)$ natural.
\end{lem}
\begin{proof}

Reading the definition using \ref{complexity of <=00003D}.
\end{proof}
\begin{thm}

\label{main theorem section 2}\foreignlanguage{british}{For every
$x\in X$ there is a natural number $m$ such that $[x]=\{y\ :\ y\equiv_{\delta(x)+m}x\}.$}
\end{thm}
\selectlanguage{british}%
\begin{proof}

Immediate from the previous Lemma and Proposition \ref{main proposition Scott}.

\end{proof}

\section{{\large Hjorth Analysis and Borel Orbit Equivalence Relations}}

We will now discuss the complexity of sets of the form $B\cdot x$
for $B$ Borel. This discussion, apart of being interesting in its
own, will be applied to the theory of Hjorth analysis.

The following is trivial:

\selectlanguage{english}%
\begin{prop}

For $B$ Borel, $B\cdot x$ is analytic.

\end{prop}

Unfortunately, we can't do better:

\begin{example}

(Hrushovski): Let $2^{\omega}\times2^{\omega}$ act on $2^{\omega}$
by $(x,y)\cdot z=x\cdot z$. Then for any $A\subseteq2^{\omega}\times2^{\omega}$
, $A\cdot1$ is the projection of $A$ on the first coordinate. Hence,
$\{F\cdot1\ :\ F\subseteq2^{\omega}\times2^{\omega}\ closed\}$ is
the collection of analytic subsets of $2^{\omega}.$ In particular,
$B\cdot x$ for $B$ Borel is not necessarily Borel.

\end{example}
\begin{prop}

$B\cdot x$ is Borel if and only if $B\cdot G_{x}$ is Borel. In particular,
$U\cdot x$ is Borel, for $U$ open.

\begin{proof}

We will give 2 different proofs:

\begin{enumerate}

\item \textbf{$B\cdot x=B\cdot G_{x}\cdot x$ , }so we may assume that $B$
is a collection of cosets of $G_{x}$. $y\notin B\cdot x\iff y\notin G\cdot x\vee\left(\exists g\ g\cdot x=y\wedge g\cdot G_{x}\cap B=\emptyset\right)\iff y\notin G\cdot x\vee\left(\exists g\ g\cdot x=y\wedge g\cdot G_{x}\nsubseteq B\right)\iff y\notin G\cdot x\vee\left(\exists g\exists h\ g\cdot x=y\wedge h\cdot x=x\wedge g\cdot h\notin B\right)$.
Hence $B\cdot x$ is co-analytic as well.
\item Consider $\pi:G\to G/G_{x}$ the canonical projection and $\phi:G/G_{x}\to G\cdot x$
the bijection between the cosets of $G_{x}$ and the orbit of $x$.
$\phi$ is a continuous bijection between standard Borel spaces, and
hence is a Borel isomorphism of the two. Hence, $B\cdot x=\phi(\pi(B))$
is Borel if and only if $\pi(B)$ is Borel, if and only if $\pi^{-1}(\pi(B))=B\cdot G_{x}$
is Borel.

\end{enumerate}
\end{proof}
\end{prop}

More can be said about the complexity of $U\cdot x$ for $U$ open:

\selectlanguage{british}%

\selectlanguage{english}%

\begin{prop}
\label{Complexity of Ux relative with Gx}

\begin{enumerate}
\item If $G\cdot x$ is $\mathbf{\Pi_{\alpha+1}^{0}}$ for $\alpha\geq1$
then for every open $U$, $U\cdot x$ is $\mathbf{\Pi_{\alpha+1}^{0}}$.
\item If $G\cdot x$ is $\mathbf{\Sigma_{\alpha}^{0}}$
then for every open $U$, $U\cdot x$ is $\mathbf{\Sigma_{\alpha}^{0}}.$
\end{enumerate}
\end{prop}
\begin{proof}

\begin{enumerate}
\item We deal first with the case $\alpha=1$, which is, $G\cdot x$ is $G_{\delta}.$
In this case, Effros' theorem is valid, and since $U\cdot G_{x}$
is open in $G\backslash G_{x}$, $U\cdot x$ is open in $G\cdot x$.
We deduce that $U\cdot x$ is $G_{\delta}$.

For arbitrary $\alpha$, there is a sequence $\langle B_{n}\ :\ n\in\omega\rangle$
of $\mathbf{\Sigma_{\alpha}^{0}}$ sets such that $G\cdot x=\bigcap_{n\in\omega}B_{n}.$
We use Vaught transforms:
\[
G\cdot x=(G\cdot x)^{*}=\bigcap_{n\in\omega}(B_{n})^{*}=\bigcap_{n\in\omega}\bigcap_{m\in\omega}(B_{n})^{\triangle U_{m}}
\]
 where $U_{m}$ is a countable basis for the topology of $G$. We
can then apply Theorem \ref{thm:(-Hjorth-) better topology} and refine
the topology $\tau$ of $X$ to a topology $\sigma$ in which $G\cdot x$
is $G_{\delta}.$ Using the case $\alpha=1$, $U\cdot x$ is $G_{\delta}$
in $\sigma$ , and hence $U\cdot x$ was $\mathbf{\Pi_{\alpha+1}^{0}}$
in the original topology.
\item By Theorem \ref{thm:(-Hjorth-) better topology}, there is $\tau\subseteq\sigma\subseteq\mathbf{\Sigma_{\alpha}^{0}(X,\tau)}$
Polish s.t. $(X,\sigma)$ is a Polish G - space and $G\cdot x$ is
open in $\sigma$. Then $U\cdot x$ is open in $(G\cdot x,\sigma)$,
so it is an intersection of 2 $\mathbf{\Sigma_{\alpha}^{0}}$ sets.
\end{enumerate}
\end{proof}

\smallskip{}
The first clause of the previous proposition is not true in general for $\alpha=0$ - one trivial example is the action of $(\mathbb{R},+)$ on itself.\\

We will now apply the above to the theory of Hjorth analysis, keeping
our main goal in mind - proving that if $E_{G}^{X}$ is Borel then
Hjorth rank must be bounded.

\selectlanguage{british}%
\begin{prop}

\label{equivalent definition <=00003Dalpha for every alpha}$y\in(V\cdot x)^{*W}$
if and only if $(y,W)\leq_{\alpha}(x,V)$ for every $\alpha<\omega_{1}$.
In particular, if $(y,W)\leq_{\alpha}(x,V)$ for every $\alpha$ then
$y$ and $x$ are orbit equivalent.

\begin{proof}

We first assume that $y\in(V\cdot x)^{*W}.$ We will show that for
every $\alpha$, $(y,W)\leq_{\alpha+1}(x,V)$. So let $W_{0}\subseteq W$.
There is $g\in W_{0}$ and $h\in V$ such that
\[
g\cdot y=h\cdot x.
\]
We then find $V_{0}\subseteq V$ small enough around $h$ that $V_{0}\cdot h^{-1}\cdot g\subseteq W_{0}$.
Hence
\[
(x,V_{0})\leq_{\alpha}(g^{-1}\cdot h\cdot x,V_{0}\cdot h^{-1}\cdot g)\leq_{\alpha}(y,W_{0})
\]
where we have used Lemma \ref{lemma invariance under g}.

For the other direction , by the previous proposition there is $\alpha$
such that $V\cdot x$ is $\mathbf{\Pi_{\alpha}^{0}}.$ Since $x\in(V\cdot x)^{*V}$
is a triviality and $(y,W)\leq_{\alpha}(x,V)$ is part of the assumption,
Lemma \ref{Main lemma section 2} implies $y\in(V\cdot x)^{*W}$.
\end{proof}
\begin{prop}

\label{(y,W)<=00003D(x,v) for all alpha}Let $x\in X$ and $\alpha$
such that for every $V\subseteq G$ open, $V\cdot x$ is $\mathbf{\Pi_{\alpha}^{0}}$.
Then $(y,W)\leq_{\alpha}(x,V)$ implies
\[
\forall\beta:\ (y,W)\leq_{\beta}(x,V).
\]

\end{prop}
\end{prop}
\begin{proof}

Immediate using the above.
\end{proof}
The boundedness principle follows easily:
\selectlanguage{american}%
\begin{thm}

\label{Becker Kechris Theorem}Let $(G,X)$ be a Polish action and
$\mathbb{B}\subset X$ an invariant Borel set. Then $E_{G}^{\mathbb{B}}$
is Borel if and only if there is an $\alpha$ such that for every
$x\in\mathbb{B}$, $\delta(x)\leq\alpha.$
\end{thm}
\selectlanguage{british}%
\begin{proof}

We first assume that the rank is bounded. We use \ref{main theorem section 2}
and Sami's theorem to show that $E_{G}^{\mathbb{B}}$ is Borel.

Now assume $E_{G}^{\mathbb{B}}$ is Borel. There is an $\alpha<\omega_{1}$
such that all orbits are $\mathbf{\Pi_{\alpha+1}^{0}}.$ Proposition \ref{Complexity of Ux relative with Gx} then implies that for all $U \subseteq G$ open, $U\cdot x$ is $\mathbf{\Pi_{\alpha+1}^{0}}$. Using the previous proposition, $\delta(x)\leq\alpha+1$.
\end{proof}
\selectlanguage{english}%
\begin{cor}

Let $O\subseteq G$ be a clopen subgroup of $G$. Then if $E_{G}^{X}$
is Borel then so does $E_{O}^{X}$. Furthermore, if Hjorth ranks of
$E_{G}^{X}$ are bounded by $\alpha$, then so do the Hjorth ranks
of $E_{O}^{X}.$
\end{cor}
\begin{proof}

If $E_{G}^{X}$ is Borel, then the complexities of $U\cdot x$ are
all less than some $\alpha.$ The open sets of $O$ are open sets
of $G$ , so the same $\alpha$ bounds the complexities of $U\cdot x$
for $U\subseteq O$, and in particular $E_{O}^{X}$ is Borel.
\end{proof}
\begin{rem}

The above is not true when $G$ is not open : there are $H\leq G$
Polish and $X$ a Polish $G$ - space, such that $E_{G}^{X}$ is Borel
and $E_{H}^{X}$ is not Borel.

\end{rem}
\selectlanguage{british}%

The notion of Hjorth rank simplifies the proof of the following theorem
of \cite{key-11} and adds information about the decomposition, although
some definability is lost on the way:

\selectlanguage{english}%
\begin{thm}

(Decomposition of Polish actions) Let $X$ be a Polish G - Space.
There is a sequence $\{A_{\zeta}\}_{\zeta<\omega_{1}}$of pairwise
disjoint Borel subsets of $X$ such that:
\end{thm}
\begin{enumerate}

\item $A_{\zeta}$ is invariant, and $\bigcup_{\zeta<\omega_{1}}A_{\zeta}=X$.
Furthermore, $A_{\zeta}$ is $\mathbf{\Pi_{\zeta+n(\zeta)}^{0}.}$
\item $E_{a}\restriction A_{\zeta}$ is Borel. In fact, it is $\mathbf{\Pi_{\zeta+k(\zeta)}^{0}}$.
\item (Boundedness) If $A\subseteq X$ is invariant Borel and $E_{a}\restriction A$
is Borel, then $A\subseteq\bigcup_{\zeta<\alpha}A_{\zeta}$ for some
$\alpha<\omega_{1}.$
\end{enumerate}
\begin{proof}

$A_{\zeta}$ will be the set of $x$'s with rank $\zeta.$

\end{proof}
\selectlanguage{british}%

We now have all that is needed to give a positive answer to a conjecture of Hjorth:

\begin{thm}

For $\beta$ limit, the set
\[
\mathbb{A}_{\beta}=\{x\ :\ [x]\ is\ \mathbf{\Pi_{\alpha}^{0}\mbox{ for }\alpha<\beta\}}
\]
 is Borel. Furthermore, it is $\mathbf{\Pi_{\beta+m}^{0}}$ for $m\in\omega$.
\end{thm}
\begin{proof}

We claim that this set is in fact $\{x\ :\ \delta(x)<\beta\}$. One
direction is Theorem \ref{main theorem section 2}. The other is immediate
given Proposition \ref{Complexity of Ux relative with Gx}.
\end{proof}
\selectlanguage{english}%
\begin{cor}

For $\alpha$ limit , there are either countably many or perfectly
many$\ \mathbf{\Pi_{\alpha}^{0}}$ orbits

\selectlanguage{british}%
\begin{proof}

Consider the action of $G$ on $\mathbb{A}_{\alpha}$ as above. The set $\mathbb{A}_{\alpha}$
is Borel invariant and the Hjorth ranks are bounded on $\mathbb{A}_{\alpha}.$
Hence the orbit equivalence relation on $\mathbb{A}_{\alpha}$ is Borel,
and there are either countably many or perfectly many orbits in $\mathbb{A}_{\alpha}$.

\end{proof}
\end{cor}
\selectlanguage{american}%

The following is a generalization of Theorem \ref{Becker Kechris 7.1.4}:

\begin{cor}

The following are equivalent:

\begin{enumerate}

\item $E_{G}^{\mathbb{B}}$ is Borel.
\item $\{(x,y,U,W)\ :\ U,W\in\mathfrak{B}_{0};\ \forall\alpha<\omega_{1}\ (x,U)\leq_{\alpha}(y,W)\}$
is Borel.
\item $\{(x,U,W)\ :\ U,W\in\mathfrak{B}_{0};\ \forall\alpha<\omega_{1}\ (x,U)\leq_{\alpha}(x,W)\}$
is Borel.

\begin{proof}

\textbf{$(1)\Rightarrow(2)$ : }Follows easily from the proof of Proposition
\ref{(y,W)<=00003D(x,v) for all alpha}.

$(2)\Rightarrow(3):$ Immediate.

$(3)\Rightarrow(1):$ Using \ref{Becker Kechris 7.1.2}, it is enough to
show that
\[
f:X\to\mathcal{F}(G)
\]
\[
f(x)=G_{x}
\]
 is Borel.

Hence, it suffices to show that for any $U\subseteq G$ , the set
\[
Z=\{x\ :\ \exists g\in U\ g\cdot x=x\}
\]
is Borel. We claim that
\[
Z=\{x\ :\ \exists V,W\in\mathfrak{B}_{0}\ s.t.\ W^{-1}\cdot V\subseteq U\ and\ \forall\alpha\ (x,V)\leq_{\alpha}(x,W)\},
\]
 which is a Borel set by the assumption, so we only need to prove
this claim.

Assume $V,W\in\mathfrak{B}_{0}$ are such that $W^{-1}\cdot V\subseteq U$
and $\forall\alpha\ (x,V)\leq_{\alpha}(x,W)$ . By Proposition \ref{equivalent definition <=00003Dalpha for every alpha},
there are $g\in V$ and $h\in W$ such that $g\cdot x=h\cdot x$.
Hence $h^{-1}\cdot g\cdot x=x$ so that $x\in Z$.

\selectlanguage{british}%
Now let\foreignlanguage{american}{ $x\in Z$, and let $g\in U$ such
that $g\cdot x=x$. For every $\alpha$ and every open set $V$:
\[
(x,V)\leq_{\alpha}(x,V\cdot g^{-1}).
\]
Thus, it will be enough to find $W,V\in\mathfrak{B}_{0}$ such that
$V\cdot g^{-1}\subseteq W$ and $W^{-1}\cdot V\subseteq U$. }

\selectlanguage{american}%
We will find $W$ a nighbourhood of the identity small enough so that
$W^{-1}\cdot W\cdot g\subseteq U$, and $V$ a neighbourhood of $g$
small enough so that $V\cdot g^{-1}\subseteq W$.

\end{proof}
\end{enumerate}
\end{cor}

A last characterization of Borel equivalence relations we mention
is the following:

\selectlanguage{english}%
\begin{prop}

\label{delta(x)=00003Ddelta(y) Borel}Let $X$ be a Polish $G$ -
space. $E_{G}^{X}$ is Borel if and only if $R_{=}=\{(x,y)\ :\ \delta(x)=\delta(y)\}$
is Borel if and only if $R_{<}=\{(x,y)\ :\ \delta(x)<\delta(y)\}$
is Borel.
\end{prop}
\begin{proof}

Assume $E_{G}^{X}$ is Borel. Then the Hjorth ranks are bounded, say
by $\gamma.$ Hence, $\delta(x)=\delta(y)$ if and only if there is
$\alpha\leq\gamma$ such that $\delta(x)=\alpha$ and $\delta(y)=\alpha$.
A similar argument works for $R_{<}.$

On the other hand, assume $E_{G}^{X}$ is not Borel. So the equivalence
relation $R_{=}$ has $\aleph_{1}$ equivalence classes, and the well
founded relation $R_{<}$ is of height $\omega_{1}$. Thus $R_{<}$
has uncountable height, so it cannot be analytic, let alone Borel.
As for $R_{=}$ , if it were Borel, there would be a perfect set of
different rank elements. Extending to a $V[G]$ such that $V[G]\models\neg CH$,
by Shoenfield's absoluteness, $R_{=}$ will still have a perfect set
of different rank elements, which is a contradiction.

\end{proof}

In the next section we will see what can be said about complexities
of rank comparisons in general.

\subsection{The Logic Action Example Revisited}

\selectlanguage{american}%

We'll consider now the logic action and see how do Hjorth and Scott
analyses compare. As a corollary, we will get Theorem \ref{Becker Kechris Logic case}
of Becker and Kechris.

Let $\mathcal{L}$ be a countable language, $Mod(\mathcal{L})$ the
Polish space of countable models of $\mathcal{L}$, and $S_{\infty}$
``logically'' acts on $Mod(\mathcal{L})$, as described in the introduction.

\begin{defn}

For $\bar{a},\bar{b}$ finite 1-1 same length sequences of natural
numbers:

\end{defn}

\[
V_{\bar{a},\bar{b}}=\{\sigma\in S_{\infty}\ :\ \sigma(\bar{a})=\bar{b}\}.
\]

The $V_{\bar{a},\bar{b}}$ sets form a countable basis of the topology
of $S_{\infty}$.

\begin{lem}

\label{characterization of closure in logic action}For $\mathcal{M}$
a countable model: $\overline{V_{\bar{a},\bar{b}}\cdot\mathcal{M}}=\{\mathcal{N}\ :\ Th_{\Sigma}(\mathcal{N},\bar{b})\subseteq Th_{\Sigma}(\mathcal{M},\bar{a})\}$
($Th_{\Sigma}$ stands for the existential first order theory of
a model) .
\end{lem}
\begin{proof}

$\mathcal{N}\in\overline{V_{\bar{a},\bar{b}}\cdot\mathcal{M}}$ if
and only if for every atomic formula $\phi(\bar{x},\bar{y})$, if
$\mathcal{N}\models\phi(\bar{b},\bar{c})$ then there is $\sigma\in V_{\bar{a},\bar{b}}$
such that $\mathcal{M}\models\phi(\bar{a},\sigma^{-1}(\bar{c})).$
In other words, if $\mathcal{N}\models\exists\bar{y}\phi(\bar{b},\bar{y})$
then $\mathcal{M}\models\exists\bar{y}\phi(\bar{a},\bar{y})$ . $\square$
\end{proof}
\begin{prop}

\label{comparison =00003D=00003D=00003DScott <=00003DHjorth}If $(\mathcal{M},\bar{a})\equiv_{\omega\cdot\alpha}(\mathcal{N},\bar{a'})$
then for every $\bar{b}$ finite 1-1 sequences:

\end{prop}

\[
(\mathcal{M},V_{\bar{a},\bar{b}})\leq_{\alpha}(\mathcal{N},V_{\bar{a'},\bar{b}}).
\]

\begin{proof}

By induction on $\alpha$. For $\alpha=1$ , we assume $(\mathcal{M},\bar{a})\equiv_{\omega}(\mathcal{N},\bar{a'})$
and want to show that
\[
\overline{V_{\bar{a},\bar{b}}\cdot\mathcal{M}}\subseteq\overline{V_{\bar{a'},\bar{b}}\cdot\mathcal{N}}
\]
 for any $\bar{b}$ as above. So assume $\mathcal{P}\in\overline{V_{\bar{a},\bar{b}}\cdot\mathcal{M}}$.
Then by Lemma \ref{characterization of closure in logic action},
$Th_{\Sigma}(\mathcal{P},\bar{b})\subseteq Th_{\Sigma}(\mathcal{M},\bar{a})$
. Then $Th_{\Sigma}(\mathcal{P},\bar{b})\subseteq Th_{\Sigma}(\mathcal{N},\bar{a'})$
, as they are logically equivalent, so that $\mathcal{P}\in\overline{V_{\bar{a'},\bar{b}}\cdot\mathcal{N}}$.

The case $\alpha$ limit is trivial. Consider then the case $\alpha=\beta+1$
. We assume $(\mathcal{M},\bar{a})\equiv_{\omega\cdot\beta+\omega}(\mathcal{N},\bar{a'})$.
Let $V_{\bar{a}\bar{c},\bar{b}\bar{d}}\subseteq V_{\bar{a},\bar{b}}$.
It will be enough to find $\bar{c'}$ such that
\[
(\mathcal{N},V_{\bar{a'}\bar{c'},\bar{b}\bar{d}})\leq_{\beta}(\mathcal{M},V_{\bar{a}\bar{c},\bar{b}\bar{d}}).
\]
By the definition of $\equiv_{\omega\cdot\beta+\omega}$ , there is
$\bar{c'}$ such that $(\mathcal{N},\bar{a'}\bar{c'})\equiv_{\omega\cdot\beta}(\mathcal{M},\bar{a}\bar{c})$.
Then by the induction hypothesis we get the above.
\end{proof}
\begin{cor}

Let $\mathbb{B}\subseteq Mod(\mathcal{L})$ be an invariant Borel
subset. Let $\alpha<\omega_{1}$ be such that for every $\mathcal{M}\in\mathbb{B}$,
the Hjorth rank $\delta_{H}(\mathcal{M})$ is bounded by $\alpha$.
Then the Scott rank $\delta_{S}(\mathcal{M})$ is bounded by $\omega\cdot(\alpha+1)$.
\end{cor}
\begin{proof}

Let $\mathcal{M}\in\mathbb{B}$, and assume $(\mathcal{M},\bar{a})\equiv_{\omega\cdot(\alpha+1)}(\mathcal{M},\bar{b})$.
Then by the previous proposition:
\[
(\mathcal{M},V_{\bar{a},\bar{c}})\leq_{\alpha+1}(\mathcal{M},V_{\bar{b},\bar{c}}).
\]

Then by the assumption, for every $\gamma<\omega_{1}$ , $(\mathcal{M},V_{\bar{a},\bar{c}})\leq_{\gamma}(\mathcal{M},V_{\bar{b},\bar{c}})$,
which is why there are $g\in V_{\bar{a},\bar{c}}$ and $h\in V_{\bar{b},\bar{c}}$
such that $g\cdot\mathcal{M}=h\cdot\mathcal{M}.$ The permutation
$h^{-1}\cdot g\in V_{\bar{a},\bar{b}}$ shows that $(\mathcal{M},\bar{a})\simeq(\mathcal{M},\bar{b})$.

\end{proof}
\begin{cor}

(Becker - Kechris) Let $\mathbb{B}\subseteq Mod(\mathcal{L})$ be
an invariant Borel subset. If $E_{G}^{\mathbb{B}}$ is Borel then
there is an $\alpha<\omega_{1}$ such that the Scott ranks of the
elements of $\mathbb{B}$ are all bounded by $\alpha$.

\end{cor}

\section{Hjorth Rank and Computable Ordinals}

\selectlanguage{british}%

\selectlanguage{english}%
Nadel \cite{key-8} has shown that for the logic action, the
Scott rank of a model $M$ is at most $\omega_{1}^{ck(M)}$. The model
$M$ is identified with a sequence of $0$'s and $1$'s the moment
we begin to talk about $Mod(L)$ as a topological space, so the meaning
of $\omega_{1}^{ck(M)}$ is clear.

In the general Polish action case, $\omega_{1}^{ck(x)}$ has no obvious
meaning a priori. However, we would like it to be the first ordinal
not computable in an oracle that knows all about the action of $G$
on that specific $x\in X$. The following definition follows:
\begin{defn}

Fix a basis $\mathfrak{B}_{0}=\langle V_{k}\ :\ k<\omega\rangle$
for the topology of $G$ and a basis $\langle U_{l}\ :\ l<\omega\rangle$
for the topology of $X$. For $x\in X$, define $x_{G}:\omega\to2$
as follows:
\[
x_{G}(\langle k,l\rangle)=1\iff(V_{k}\cdot x)\cap U_{l}\neq\emptyset.
\]
 $x_{G}$ codes the action of $G$ on $X$. The map $x\to x_{G}$
is a Borel map from $X$ to $2^{\omega}$. We'll usually abuse notation
and write $x$ instead of $x_{G}.$ In particular, $\omega_{1}^{ck(x)}$
stands for $\omega_{1}^{ck(x_{G})}.$

\end{defn}

We show Nadel's theorem for Hjorth analysis:

\begin{thm}

\label{thm:Nadel's}For every $x\in X$, $\delta(x)\leq\omega_{1}^{ck(x)}.$
In particular, the orbit of $x$ is $\mathbf{\Pi_{\omega_{1}^{[x]}+n}^{0}}$,
for some $n\in\omega,$ where $\omega_{1}^{[x]}=min\{\omega_{1}^{y}\ :\ yEx\}.$

\end{thm}

For the proof, we first analyze the lightface complexity of $\leq_{\alpha}$.
First some ad-hoc definitions:

\begin{defn}

We say that $k^{*}$ is \textbf{contained} in $k$ if $V_{k^{*}}\subseteq V_{k}$.
We say that $k^{*},m^{*}\in\omega$ are \textbf{fine} with respect
to $k,m\in\omega$ if $\overline{V_{k^{*}}}\subseteq V_{k}$ and $\overline{V_{m}}\subseteq V_{m^{*}}.$
We assume that both these relations are recursive.

We define the relation $S$ as follows:

For $x_{G}\in2^{\omega}$ as above, $z\in2^{\omega},$ $R\in LO$:
$(x_{G},z,R)\in S\iff$
\end{defn}

\[
\left(\forall k,m:\ z_{\langle n,k,m\rangle}=1\iff(x,V_{k})\leq_{tp(n)\ w.r.t.\ R}(x,V_{m})\right)
\]

$S$ says that $z$ codes all the information about $\leq_{\alpha}$
on $x$ for $\alpha$'s smaller than $tp(R)$. If $R$ is not a well
order, we don't care what $S$ means.

\begin{lem}

$S$ is arithmetic.
\end{lem}
\selectlanguage{british}%
\begin{proof}

Easy.

\selectlanguage{english}%
\begin{proof}

(of Theorem \ref{thm:Nadel's}): Denote $\omega_{1}^{ck(x)}$ by
$\epsilon.$ Assume by way of contradiction that $\delta(x)>\epsilon.$
Then there are $k,m$ and $k^{*},m^{*}$ fine with respect to $k,m$
such that:
\[
(x,V_{k})\leq_{\epsilon}(x,V_{m})
\]
but
\[
(x,V_{k^{*}})\nleq_{\epsilon+1}(x,V_{m^{*}})
\]
so there is $i$ contained in $k^{*}$ such that for every $j$ contained
in $m^{*}:$
\[
(x,V_{j})\nleq_{\epsilon}(x,V_{i})
\]
For each $j$ as above, we'll choose $\alpha_{j}<\epsilon$ such that
$(x,V_{j})\leq_{\alpha_{j}}(x,V_{i})$ but $(x,V_{j})\nleq_{\alpha_{j}+1}(x,V_{i}).$

\begin{claim}

$\alpha_{j}$ are cofinal in $\epsilon.$
\end{claim}
\begin{proof}

Otherwise, there is $\beta<\epsilon$ such that all $\alpha_{j}$'s
are less than $\beta.$ But $(x,V_{k^{*}})\leq_{\epsilon}(x,V_{m^{*}})$
so there is $j$ contained in $m^{*}$ such that $(x,V_{j})\leq_{\beta+1}(x,V_{i})$.
Which is a contradiction.

\end{proof}
\end{proof}

However:

\begin{claim}

The set $W=\{(j,R)\ :\ tp(R)=\alpha_{j};\ j\ is\ contained\ in\ m^{*}\}$
is $\Sigma_{1}^{1}(x)$.

\begin{proof}

$(j,R)\in W$ $\iff$ all of the following are satisfied:

\begin{enumerate}

\item $R$ is a linear order computable by $x$. $j$ is contained in $m^{*}$.
\item There exist a program $e$, natural numbers $n,l$ and a sequence
$z$ such that $[e](x)$ is $tp(R)+2$ , $n$ is the last element
of $[e](x)$, $l$ is its immediate predecessor, $(x,z,[e](x))\in S$
, $z_{\langle l,j,i\rangle}=1$ and $z_{\langle n,j,i\rangle}=0.$
\item For every program $e$, if $[e](x)$ is a well order,\textbf{ then}
there is a program $e'$, a sequence $z$, and a natural number $n$
such that $[e'](x)$ is $[e](x)+1$, $n$ is the last element of $[e'](x)$,
$(x,z,[e'](x))\in S$ , and if $z_{\langle n,j,i\rangle}=0$ then
there is an embedding $i:(\omega,R)\to(\omega,[e](x))$, order preserving
and onto a proper initial segment of $[e](x)$.

\end{enumerate}
\end{proof}
\end{claim}

The last 2 claims contradict the boundedness theorem.

\end{proof}
\selectlanguage{english}%

We next generalize Sacks' theorem:

\begin{thm}

\label{theorem sacks}If for every $x$, $\delta(x)\leq\alpha$ or
$\delta(x)<\omega_{1}^{ck(x)}$ , then Hjorth ranks are bounded.

\end{thm}
\selectlanguage{british}%

First a useful lemma:

\selectlanguage{english}%
\begin{lem}

\label{lem:good and bad}There are $\Sigma_{1}^{1}$ and $\Pi_{1}^{1}$
formulas $\phi(x,R)$ and $\psi(x,R)$ such that if $R\in WO$ then
$\phi(x_{G},R)$ if and only if $\psi(x_{G},R)$ if and only if
\[
\forall k\forall m\left(\forall k^{*}\forall m^{*}\ fine\ w.r.t.\ \ k,m\right)\ \ (x,V_{k})\leq_{tp(R)}(x,V_{m})\Rightarrow(x,V_{k^{*}})\leq_{tp(R)+1}(x,V_{m^{*}})
\]
if and only if $\delta(x)\leq tp(R).$

\begin{proof}

Consider the following $\Sigma_{1}^{1}$ formula:

There exist a program $e$, natural numbers $n,l$ and a sequence
$z$ such that $[e](x)$ is $tp(R)+2$, $(x,z,[e](x))\in S$, $n$
is the last element in $[e](x)$ , $l$ is its immediate predecessor,
and for all $k,m$ and for all $k^{*},m^{*}$ fine with respect to
$k,m,$ if $z_{\langle l,k,m\rangle}=1$ then $z_{\langle n,k^{*},m^{*}\rangle}=1$.

Under the above conditions, this can also be written by a $\Pi_{1}^{1}$
formula.

\end{proof}
\end{lem}
\begin{proof}

\textit{(of Theorem \ref{theorem sacks})} Fix $y$ that computes
$\alpha.$ So for every $x$, $\delta(x)<\omega_{1}^{ck(x,y)}.$ Denote
by $H$ the set
\[
\{(R,x)\ :\ tp(R)=\delta(x)\}
\]
We claim that $H$ is $\Sigma_{1}^{1}(y)$:

$(R,x)\in H$ if and only if all of the following are satisfied:

\begin{enumerate}

\item $R$ is a linear order computable by $x,y$.
\selectlanguage{british}%
\item $\delta(x)\leq tp(R).$
\selectlanguage{english}%
\item For every program $e$, if $e[x,y]$ computes a well order $Q$ and
$\delta(x)\leq tp(Q)$, then there is $i:(\omega,R)\to(\omega,Q),$
order preserving and onto an initial segment of $Q$.

\end{enumerate}

Using Lemma \ref{lem:good and bad}, $2$ and $3$ are $\Sigma_{1}^{1}(y)$.

By the boundedness theorem, the Hjorth ranks are bounded.

\end{proof}

We apply the above to compute the complexity of rank comparisons:

\begin{cor}

\label{cor:BP}$\{(x,y)\ :\ \delta(x)=\delta(y)\}$ is analytic. $\{(x,y)\ :\ \delta(x)<\delta(y)\}$
is $\mathbf{\Pi_{1}^{1}}$ . $\{(x,y)\ :\ \delta(x)\leq\delta(y)$\}
is analytic.

In light of Proposition \ref{delta(x)=00003Ddelta(y) Borel}, these
are optimal.

\begin{proof}

We remind that $x\to x_{G}$ is Borel.

$\delta(x)=\delta(y)$ if and only if for every program $e$, if $e[x_{G},y_{G}]$
computes a well order $R$, then $\delta(x)\leq tp(R)\iff\delta(y)\leq tp(R)$.
We now use Lemma \ref{lem:good and bad}.

$\delta(x)<\delta(y)$ if and only if there is a program $e$ such
that $e[y_{G}]$ computes a well order $R$, and $\delta(x)\leq tp(R)$
but $\delta(y)>tp(R)$. Again, we use Lemma \ref{lem:good and bad}.

\end{proof}
\end{cor}

Which leads us to the following:

\begin{cor}

\label{cor:non meager A}Let $X$ be a perfect Polish $G$ - space,
$A_{\alpha}=\{x\ :\ \delta(x)=\alpha\}.$ There is an $\alpha<\omega_{1}$
such that $A_{\alpha}$ is non meager.
\end{cor}
\begin{fact}

\label{fact from kechris}(Kechris \cite{key-9}) Let $X$ be a
Polish space, $A\subseteq X$, $\langle A_{\alpha}\ :\ \alpha\in\omega_{1}\rangle$
a partition of $A$ into a family of disjoint meager sets in $X$.
Let $\leq^{*}$ be defined by: $x\leq^{*}y$ if both $x,y$ are in
$A$, and the $\alpha$ such that $x\in A_{\alpha}$ is smaller or
equal than the $\beta$ such that $y\in A_{\beta}.$ Then if $\leq^{*}$
has the $BP$, then $A$ is meager.
\end{fact}
\begin{proof}

\textit{(of the Corollary)} Using the fact and Corollary \ref{cor:BP}.
\end{proof}
\begin{cor}

\label{cor:non_meager_orbits}Let $X$ be a Polish $G$ - space which
is a counterexample to Vaught conjecture. Let $Y\subseteq X$ be any
non meager set. Then $Y$ has a non meager orbit. Hence, the union
of all meager orbits must be meager.
\end{cor}
\begin{proof}

$Y=\bigcup_{\alpha<\omega_{1}}(Y\cap A_{\alpha})$. Using Fact \ref{fact from kechris}
again, one of the $Y\cap A_{\alpha}$ is non meager. But it has only
countably many orbits, so one of them is non meager.
\end{proof}
\begin{cor}

Let $X$ be a counterexample to Vaught. Then there is a non empty
and at most countable collection $\langle C_{i}\ :\ i\in\omega\rangle$
of $G_{\delta}$ orbits, some of them non meager, such that $C=\bigcup_{i\in\omega}C_{i}$
is comeager. At least one of those orbits is not $F_{\sigma}$.
\end{cor}
\begin{proof}

By the previous corollary, there are non meager orbits. All of them
must be $G_{\delta}$ (by Theorem \ref{thm:(-Effros-)}) and there
are at most countably many $G_{\delta}$ orbits. Fix $C_{i}$ an enumeration
of the $G_{\delta}$ orbits (which contains all the non meager ones). By the previous corollary, $C$ is comeager.

Now, if all these orbits are $F_{\sigma}$ , then $C$ is $F_{\sigma}.$
Hence we can consider the action of $G$ on the Polish space $X-C$.
It will also be a counterexample to Vaught, so it must have a non
meager orbit. But a non meager orbit is $G_{\delta}$ , which is a
contradiction.
\end{proof}

\end{onehalfspace}
\end{document}